\documentclass{amsart}
\usepackage{amsmath}
\usepackage{amssymb}
\usepackage{amsfonts} 

\usepackage[hypertex,colorlinks=true,linkcolor=blue,citecolor=blue,]{hyperref}

\relax
\PassOptionsToClass{twoside}{article}

\RequirePackage{amsfonts,amssymb,amsmath,amscd,amsthm}
\RequirePackage{txfonts}
\RequirePackage{graphicx}
\RequirePackage{xcolor}
\RequirePackage{enumerate}

\newtheorem{theorem}{Theorem}[section]
\newtheorem{lemma}[theorem]{Lemma}
\newtheorem{corollary}[theorem]{Corollary}
\newtheorem{proposition}[theorem]{Proposition}

\theoremstyle{definition}
\newtheorem{definition}[theorem]{Definition}

\theoremstyle{remark}
\newtheorem{remark}[theorem]{Remark}

\numberwithin{equation}{section}

\begin{document}

\newcommand{\om}{\omega}
\newcommand{\si}{\sigma}
\newcommand{\la}{\lambda}
\newcommand{\ph}{\varphi}
\newcommand{\ep}{\varepsilon}
\newcommand{\tep}{\widetilde{\varepsilon}}
\newcommand{\al}{\alpha}
\newcommand{\sub}{\subseteq}
\newcommand{\hra}{\hookrightarrow}
\newcommand{\de}{\delta}

\
\newcommand{\RR}{{\mathbb{R}}}
\newcommand{\NN}{{\mathbb{N}}}
\newcommand{\DD}{{\mathbb{D}}}
\newcommand{\CC}{{\mathbb{C}}}
\newcommand{\ZZ}{{\mathbb{Z}}}
\newcommand{\T}{{\mathbb{T}}}
\newcommand{\TT}{{\mathbb{T}}}
\newcommand{\KK}{{\mathbb{K}}}
\newcommand{\px}{\partial X}
\newcommand{\cL}{{\mathcal{L}}}
\newcommand{\cA}{{\mathcal{A}}}
\newcommand{\cB}{{\mathcal{B}}}

\def\C{\mathbb C}
\def\R{\mathbb R}
\def\X{\mathbb X}
\def\U{\mathcal U}\def\M{\mathcal M}

\def\K{\mathbb K}
\def\P{\mathbb P}
\def\a{\mathfrak{A}}
\def\b{\mathfrak{B}}
\def\f{\mathfrak{F}}
\def\x{\mathfrak{X}}
\def\Z{\mathbb Z}
\def\P{\mathbb P}
\def\v{\vartheta_\alpha}
\def\va{{\varpi}_\alpha}
\def\I{\mathbb I}
\def\H{\mathbb H}
\def\Y{\mathbb Y}
\def\E{\mathbb E}
\def\N{\mathbb N}
\def\cal{\mathcal}

\title[Existence of Pseudo-Almost Automorphic Solutions]{Existence Results for Some Nonautonomous Integro-differential Equations}

\author{Toka Diagana}
\address{Department of Mathematics, Howard University, 2441 6th
Street N.W.,  Washington, D.C. 20059, USA}

\email{tdiagana@howard.edu}
%\thanks{}

%    General info


\maketitle

\begin{center}
{\it In Memory of Prof. Dialla KONAT\'{E}}
\end{center}

\begin{abstract} In this paper we make a subtle use of tools from operator theory and the Schauder fixed-point theorem to establish the existence of pseudo-almost automorphic solutions to some classes of nonautonomous integro-differential equations with pseudo-almost automorphic forcing terms.  To illustrate our main results, the existence of pseudo-almost automorphic solutions to a parabolic Neumann boundary value problem that models
population genetics
and nerve pulse propagation 
will be discussed.
\end{abstract}

\maketitle

\section{Introduction} 
Integro-differential equations play a crucial role in qualitative theory of differential equations due to their applications to natural phenomena, see, e.g.,
 \cite{HR, G, N, NO, V1, V2, V3}.
Much work has been done in recent years to investigate  
the existence of periodic, almost periodic, almost automorphic, pseudo-almost periodic, and pseudo-almost automorphic solutions to integro-differential equations, see, e.g., \cite{AG, TDbook, DE, DE1, DE2, H0, H1, H2, li1, li2, V1, V2, V3}. 
The existence of solutions to autonomous integro-differential equations in the above-mentioned function spaces is, to some extent, relatively well-understood. The method most widely used to deal with the existence of solutions to those
integro-differential equations is the so-called `method of resolvents', see, e.g., \cite{adez, da1, da2, H1, H2, li1, li2}.

Fix $\alpha \in (0, 1)$.
The purpose of this paper consists of making use of a new approach to study the existence of pseudo-almost automorphic solutions to the class of non-autonomous Volterra integro-differential equations given by
\begin{eqnarray}\label{PR}
\frac{d\ph}{dt} = A(t) \ph  + \int_{-\infty}^t C(t-s) \ph (s) ds + f(t,\ph),
\end{eqnarray}
where $A(t) : D \subset \X \mapsto \X$ is a family of closed unbounded linear operators on a Banach space $\X$ whose domains $D(A(t)) = D$ are constant in $t \in \R$, $C(t): D \subset \X \mapsto \X$ are (possibly unbounded) linear operators upon $\X$,
and
the function $f: \R \times \X_\alpha \mapsto \X$ is pseudo-almost automorphic in the first variable uniformly in the second one with $\X_\alpha$ being the real interpolation space of order $(\alpha, \infty)$ between $\X$ and $D$.

These types of equations arise very often in the study of natural phenomena in which a certain memory effect is taken into consideration \cite{HR, G, N}. In \cite{NO} for instance, equations of type Eq. (\ref{PR}) appeared in the study of heat conduction in materials with memory. 
The existence, uniqueness, maximal regularity, and asymptotic behavior of solutions to Eq. (\ref{PR}) have widely been studied, see, e.g., \cite{AA, CC, C, F2, HR, J, Lun2, Lun4, M}. However, to the best of our knowledge, the existence of pseudo-almost automorphic solutions to Eq. (\ref{PR}) is an untreated original problem with important applications, which constitutes the main motivation of our study.

In order to investigate the existence of pseudo-almost automorphic solutions to Eq. (\ref{PR}), we study the non-autonomous abstract differential equations involving the time-dependent linear operators $A(t)$ and the function $B(\cdot): \R \mapsto {\mathcal L}(C(\R, D), \X)$, that is,
\begin{eqnarray}\label{PR1}
\frac{d\ph}{dt} = A(t)\ph + B(t) \ph + f(t,\ph),
\end{eqnarray}
where $A(t): D \subset \X \mapsto \X$ is a family of closed linear operators on $\X$ with constant domains $D$ and the linear operators $B(t): C(\R, D)\mapsto \X$ defined by 
\begin{eqnarray}\label{PR2}B(t) \ph  := \int_{-\infty}^t C(t-s) \ph (s) ds, \ \ \ph \in C(\R, D)\end{eqnarray}with $C(\R, D)$ being the collection of all continuous functions from $\R$ into $D$.

In order to study the existence of solutions to Eq. (\ref{PR1}), we will make extensive use of exponential dichotomy tools, real interpolation spaces, and suppose that for all $\ph \in PAA(\X_\alpha)$, the function $t \mapsto B(t) \varphi$ belongs to $PAA(\X)$. Our existence result will then be obtained through the use of the well-known Schauder fixed-point theorem. Obviously, once we establish the sought existence results for Eq. (\ref{PR1}), then we can easily go back to Eq. (\ref{PR}) notably through Eq. (\ref{PR2}).

The notion of pseudo-almost automorphy is a powerful concept that has been introduced and studied in a series of recent papers by Liang {\it et al.}
\cite{L, LL, XJ}. Such a concept has recently generated several developments and these and related topics have recently been summarized in the new book by Diagana \cite{TDbook}. 

Some recent contributions on almost periodic and asymptotically
almost periodic solutions to integro-differential equations of the form 
Eq. (\ref{PR}) have recently been made in \cite{H1, H2} in the case $A(t) = A$ is constant. Similarly, in \cite{li2}, the existence of pseudo-almost automorphic solutions to an autonomous version of Eq. (\ref{PR}) was studied. The main method used in the above-mentioned papers are resolvents operators.
However, to the best of our knowledge, the existence of
pseudo-almost automorphic solutions to Eq. (\ref{PR}) (and hence to Eq. (\ref{PR1})) in the case when $A(t)$ are sectorial linear operators is an original untreated topic with some interesting
applications to the real world problems. Among other things, we will make
extensive use of the so-called Acquistpace-Terreni conditions method associated with sectorial operators $A(t)$ and the Schauder fixed point to
derive some sufficient conditions for the existence of
pseudo-almost automorphic (mild) solutions to \eqref{PR1} and then to Eq. (\ref{PR}). To illustrate our
main results, the existence of pseudo-almost automorphic solutions to a parabolic
Neumann boundary value problem that models
population genetics \cite{V2, V3}
and nerve pulse propagation \cite{V1}
will be discussed.

\section{Preliminaries} The basic results discussed in this section are mainly taken from the following recent papers by the author \cite{MCM} and \cite{PAMS}.

In this paper, $(\X, \|\cdot\|)$ denotes a Banach space. If $A$ is a linear operator upon $\X$, then the notations $D(A)$, $\rho(A)$,
$\sigma(A)$, $N(A)$, and $R(A)$ stand respectively for the domain, resolvent, spectrum, kernel, and the range of $A$.
Similarly, if $A: D= D(A) \subset \X \mapsto \X$ is a closed linear operator, one denotes its graph norm by $\|\cdot\|_{D}$ defined by
$\|\ph\|_D := \|\ph\| + \|A\ph\|$ for all $\ph \in D$. From the closedness of $A$, one can easily see that $(D, \|\cdot\|_{D})$ is a Banach space.
Moreover, one sets $R(\lambda, L) := (\lambda I - L)^{-1}$ for all
$\lambda \in \rho(A)$. 
Furthermore, we set $Q(t) =I-P(t)$ for projections
$P(t)$. If $\Y, \Z$ are Banach spaces, then the space $B(\Y, \Z)$ denotes the collection of all bounded
linear operators from $\Y$ into $\Z$ equipped with its natural
uniform operator topology $\|\cdot\|_{B(\Y, \Z)}$. We also set $B(\Y) = (\Y, \Y)$ whose corresponding norm will be denoted $\|\cdot\|$. 
If $K \subset \X$ is a subset, we let $\overline{co}\, K$ denote the closed convex hull of $K$. Additionally, 
$\T$ will denote the set defined by, $\T :=\{(t,s) \in \R \times \R: t \geq s\}.$

\subsection{Evolution Families}
\begin{definition}\cite{aq, at, TDbook, MCM, PAMS} \label{UT}
A family of closed linear operators
$A(t)$ for $t\in \R$ on $\X$ with domains $D(A(t))$ (possibly not
densely defined) is said to satisfy the so-called Acquistapace--Terreni
conditions, if there exist constants $\omega\in \R$, $\theta
\in (\frac{\pi}{2},\pi)$, $K, L \geq 0$ and $\mu, \nu \in (0,
1]$ with $\mu + \nu > 1$ such that
\begin{equation}\label{AT1}
  S_{\theta} \cup \{0\} \subset \rho(A(t)-\om I),\;\qquad
  \|R(\la,A(t)-\om I)\|\le \frac{K}{1+|\la|}, \ \, \mbox{and} \end{equation}
\begin{equation}\label{AT2}
\|(A(t)-\om I)R(\la,A(t)-\om I)\, [R (\om,A(t))-R (\om,A(s))]\|
  \le L\, |t-s|^\mu\,|\la|^{-\nu}
  \end{equation}
for $t,s\in\R$, $\displaystyle \la \in S_\theta$, where $$S_\theta:=
\{\la\in\CC\setminus\{0\}: |\arg \la|\le\theta\}.$$
\end{definition}
Among other things, the Acquistapace--Terreni conditions do ensure the existence of a unique evolution
family $${\mathcal U} = \Big\{U(t,s): t, s \in \R \ \ \mbox{such that} \ \ t \geq s\Big\}$$ on $\X$ associated with $A(t)$ 
such that $U(t, s)\X \sub D(A(t))$ for all $t, s \in \R$ with $t > s$,
and

\begin{enumerate}
\item[(a)]  $U(t,s)U(s,r)=U(t,r)$ for $t,s \in \R$ such that $t \geq s \geq s$;
\item[(b)] $U(t,t)=I$ for $t \in \R$ where $I$ is the identity operator of $\X$; and
\item[(c)] for $t > s$, the mapping $(t,s)\mapsto U(t,s)\in B(\X)$ is continuous and
continuously differentiable in $t$ with $\partial_t U (t,s) = A(t)U(t, s).$ Moreover, there exists a constant $C' > 0$ which depends on constants in Eq. (\ref{AT1}) and Eq. (\ref{AT2}) such that
\begin{eqnarray}\label{INe}
\|A^k(t) U(t,s)\| \leq C' (t-s)^{-k}
\end{eqnarray}
for $0 < t-s \leq 1$ and $k = 0, 1$.
\end{enumerate}

\begin{definition}\cite{TDbook} \label{expo}
An evolution family ${\mathcal U} = \{U(t,s): (t,s) \in \T\}$ is said to have an {\it exponential
  dichotomy} if there are projections
$P(t)$ ($t\in\R$) that are uniformly bounded and strongly
continuous in $t$ and constants $\delta>0$  and $N\ge1$ such that
\begin{enumerate}
\item[(d)] $U(t,s)P(s) = P(t)U(t,s)$; \item[(e)] the restriction
$U_Q(t,s):Q(s)\X\to Q(t)\X$ of $U(t,s)$ is
  invertible (we then set $U_Q(s,t):=U_Q(t,s)^{-1}$) where $Q(t) = I - P(t)$; and
\item[(f)] $\|U(t,s)P(s)\| \le Ne^{-\delta (t-s)}$ and
  $\|U_Q(s,t)Q(t)\|\le Ne^{-\delta (t-s)}$ for $t\ge s$ and $t,s\in \R$.
\end{enumerate}
\end{definition}
If an evolution family ${\mathcal U} = \{U(t,s): (t,s) \in \T\}$  has an exponential dichotomy, we then define 
$$
\Gamma (t,s) :=
\left\{
        \begin{array}{ll}
                U(t,s) P(s),  & \mbox{if } t \geq s, \ \ t,s \in \R, \\ \\
                -U_Q(t,s) Q(s), & \mbox{if } s > t, \ \ t, s \in \R.
        \end{array}
\right.
$$

\subsection{Estimates for $U(t,s)$}
This setting requires some estimates related to $U(t,s)$. 
For
that, we make extensive use of the real interpolation spaces of order $(\alpha, \infty)$ between $\X$ and $D(A(t))$, where $\alpha \in (0, 1)$. We refer
the reader to  Amann \cite{Am} and Lunardi \cite{Lun} for proofs and further information on theses
interpolation spaces.

Let $A$ be a sectorial operator on $\X$ (for that, in Definition \ref{UT},
replace $A(t)$ with $A$) and let $\alpha\in(0,1)$. Define the real
interpolation space
$$\displaystyle \X^A_{\alpha}: = \left\{x\in \X: \left\|x\right\|^A_{\alpha}:=
\sup\nolimits_{r>0}
\left\|r^{\alpha}(A-\omega)R(r,A-\omega)x\right\|<\infty
\right\},
$$ which, by the way, is a Banach space when endowed with the norm $\left\|\cdot
\right\|^A_{\alpha}$. For convenience we further write
$$\X_0^A:=\X,
\ \left\|x\right\|_0^A:=\left\|x\right\|, \ \X_1^A:=D(A)$$ and
$$\left\|x\right\|^A_{1}:=\left\|(\omega-A)x\right\|.$$ Moreover, let
$\hat{\X}^A:=\overline{D(A)}$ of $\X$. In particular, we have the
following continuous embedding
\begin{equation} \label{embeddings1}
\begin{split}
&D(A)\hookrightarrow \X^A_{\beta}\hookrightarrow
D((\omega-A)^{\alpha}) \hookrightarrow
\X^A_{\alpha}\hookrightarrow \hat{\X}^A \hookrightarrow \X,
\end{split}
\end{equation}
for all $0<\alpha<\beta<1$, where the fractional powers are
defined in the usual way.

In general, $D(A)$ is not dense in the spaces $\X_\alpha^A$ and
$\X$. However, we have the following continuous injection
\begin{equation*}\label{closure}
\X_\beta^A \hra \overline{D(A)}^{\|\cdot\|_\alpha^A}
\end{equation*}
for $0<\alpha <\beta <1$.

Given the family of linear operators $A(t)$ for $t\in \R$,
satisfying Acquistapace--Terreni conditions, we set
$$\X^t_\alpha:=\X_\alpha^{A(t)}, \quad
\hat{\X}^t:=\hat{\X}^{A(t)}$$ for $0\le \alpha\le 1$ and $t\in\R$,
with the corresponding norms. Then the embedding in Eq.
\eqref{embeddings1} holds with constants independent of $t\in\R$.
These interpolation spaces are of class
 $\mathcal{J}_{\alpha}$ (\cite[Definition 1.1.1 ]{Lun}) and
 hence there is a constant $c(\alpha)$ such that
 \begin{equation*}\label{J}
\left\|y\right\|_{\alpha}^t\leq c(\alpha)\left\|y\right\|^{1-\alpha} \left\|A(t)y\right\|^{\alpha}, \;\;\; y\in
 D(A(t)).
 \end{equation*}

We have the following estimates for the evolution
family $U(t,s)$.

\begin{proposition}\cite{W}\label{pes} Suppose the evolution family ${\U}$ has exponential dichotomy. 
For $x \in \X$, $ 0\leq \alpha \leq 1$ and $t > s,$ the following
hold:

\begin{enumerate}
\item[(i)] There is a constant $c(\alpha),$ such that %%
 \begin{equation}\label{eq1.1}
  \left\|U(t,s)P(s)x\right\|_{\alpha}^t\leq
 c(\alpha)e^{- \frac{\delta}{2}(t-s)}(t-s)^{-\alpha} \left\|x\right\|.
  \end{equation}
%%%
\item[(ii)] There is a constant $m(\alpha),$ such that

 \begin{equation}\label{eq2.1}
 \left\|\widetilde{U}_{Q}(s,t)Q(t)x\right\|_{\alpha}^s\leq
 m(\alpha)e^{-\delta (t-s)}\left\|x\right\|.
 \end{equation}
 \end{enumerate}

\end{proposition}

\begin{remark}

Note that if an evolution family $\U$ is exponential stable, that is, there exists constants $N, \delta > 0$ such that $\|U(t,s)\| \leq N e^{-\delta(t-s)}$ for $t \geq s$, then its dichotomy projection $P(t) = I$ ($Q(t) = I - P(t) = 0$). In that case, Eq. (\ref{eq1.1}) still holds and can be rewritten as
follows: for all $x \in \X$,
\begin{equation}\label{eq1.100}
  \left\|U(t,s)x\right\|_{\alpha}^t\leq
 c(\alpha)e^{- \frac{\delta}{2}(t-s)}(t-s)^{-\alpha} \left\|x\right\|.
  \end{equation}
\end{remark}

\begin{remark}\label{RE} Note that 
if the evolution family $U(t,s)$ is compact for $t > s$ and is exponential stable, then it can be shown that for each given $t \in \R$ and $\tau > 0$, the family, $$\Big\{U(\cdot, s): \ s \in (-\infty, t-\tau)\Big\}$$is equi-continuous in $t$ for the uniform operator topology.
\end{remark}

\subsection{Pseudo-Almost Automorphic Functions}
Let $BC(\R, \X)$ stand for the Banach space of all bounded continuous functions $\ph: \R \mapsto \X$, which we equip with the sup-norm defined by
$\|\ph\|_\infty: = \sup_{t \in \R} \left\|\ph(t)\right\|$
for all $\ph \in BC(\R, \X)$.
Similarly, letting $\X_\alpha = (\X, D)_{\alpha, \infty}$ for $\alpha \in (0, 1)$, then the space $BC(\R, \X_\alpha)$ will stand for the Banach space of all bounded continuous functions $\ph: \R \mapsto \X_\alpha $, which we equip with the sup norm defined by
$\|\ph\|_{\alpha, \infty}: = \sup_{t \in \R} \left\|\ph(t)\right\|_\alpha$
for all $\ph \in BC(\R, \X_\alpha)$.

\begin{definition}\cite{TDbook} \label{DDD}
A function $f\in C(\R,\X)$ is said to be almost automorphic if for
every sequence of real numbers $(s'_n)_{n \in \N}$, there
   exists a subsequence $(s_n)_{n \in \N}$ such that
      $$ g(t):=\lim_{n\to\infty}f(t+s_n)$$
   is well defined for each $t\in\mathbb{R}$, and
      $$ \lim_{n\to\infty}g(t-s_n)=f(t)$$
   for each $t\in \mathbb{R}$.
\end{definition}

If the convergence above is uniform in $t\in \R$, then $f$ is
almost periodic in the classical Bochner's sense. Denote by
$AA(\X)$ the collection of all almost automorphic functions
$\R\mapsto \X$. Note that $AA(\X)$ equipped with the sup-norm
turns out to be a Banach space.

Among other things, almost automorphic functions satisfy the
following properties.

\begin{theorem}\cite{TDbook}\label{T}
   If $f, f_1, f_2\in AA(\X)$, then
   \begin{itemize}
      \item[(i)] $f_1+f_2\in AA(\X)$,
      \item[(ii)] $\lambda f\in AA(\X)$ for any scalar $\lambda$,
      \item[(iii)] $f_\alpha\in AA(\X)$ where $f_\alpha:\mathbb{R}\to \X$ is defined by
                     $f_\alpha(\cdot)=f(\cdot+\alpha)$,
      \item[(iv)] the range $\mathcal{R}_f:=\big\{f(t):t\in\mathbb{R}\big\}$ is relatively
                    compact in $\X$, thus $f$ is bounded in norm,
      \item[(v)] if $f_n\to f$ uniformly on $\mathbb{R}$ where each $f_n\in AA(\X)$, then $f\in
                   AA(\X)$ too.
   \end{itemize}
\end{theorem}

\begin{definition}\label{KKK}Let $\Y$ be another Banach space.
A jointly continuous function $F: \R \times \Y \mapsto \X$ is said
to be almost automorphic in $t \in \R$ if $t \mapsto F(t,x)$ is
almost automorphic for all $x \in K$ ($K \subset \Y$ being any
bounded subset). Equivalently, for every sequence of real numbers
$(s'_n)_{n \in \N}$, there
   exists a subsequence $(s_n)_{n \in \N}$ such that
      $$G(t, x):=\lim_{n\to\infty}F(t+s_n, x)$$
   is well defined in $t\in\mathbb{R}$ and for each $x \in K$, and
      $$ \lim_{n\to\infty}G(t-s_n, x)=F(t, x)$$
   for all $t\in \mathbb{R}$ and $x \in K$.

 The collection of such functions will be denoted by $AA(\R \times \X)$.
\end{definition}

We now introduce the notion of bi-almost automorphy, which in fact is due to Xiao {\it et al}. \cite{XJ}.

\begin{definition} 
A jointly continuous function $F: \T \mapsto \X$ is called positively bi-almost automorphic if for every sequence of real
numbers $(s'_n)_{n \in \N}$, we can extract a subsequence $(s_n)_{n \in \N}$ such that $$G(t,s): = \lim_{n \to \infty} F(t+s_n, s+s_n)$$ is well defined for $(t,s) \in \T$,
and $$\lim_{n \to \infty} G(t - s_n, s - s_n) = F(t, s)$$ for each $(t,s) \in \T$.

The collection of such functions will be denoted $bAA(\T, \X)$.
\end{definition}

 For
more on almost automorphic functions and their generalizations, we refer
the reader to the recent book by Diagana \cite{TDbook}.

Define
$PAP_0(\R, \X)$ as the collection of all functions $\ph \in BC(\R, \X)$ satisfying, $$\lim_{r \to \infty}
\displaystyle{\frac{1}{2r}} \int_{-r}^r \|\ph(s)\| ds =
0.$$

Similarly, $PAP_0(\R \times \X)$ will denote the collection of all
bounded continuous functions $F: \R \times \Y \mapsto \X$ such
that
$$\lim_{T \to \infty}
\displaystyle{\frac{1}{2r}} \int_{-r}^r \| F(s, x)\| ds =
0$$ uniformly in $x \in K$, where $K \subset \Y$ is any bounded
subset.

\begin{definition}\label{DEF} (Liang {\it et al.} \cite{L} and {\it Xiao et al.} \cite{LL}) 
A function $f \in BC(\R, \X)$ is called pseudo almost automorphic
if it can be expressed as $f = g + \phi,$ where $g \in AA(\X)$ and
$\phi \in PAP_0(\X)$. The collection of such functions will be
denoted by $PAA({\mathbb X})$.
\end{definition}

The functions $g$ and $\phi$ appearing in Definition~\ref{DEF} are
respectively called the {\it almost automorphic} and the {\it
ergodic perturbation} components of $f$.

\begin{definition} Let $\Y$ be another Banach space.
A bounded continuous function $F: \R \times \Y \mapsto \X$ belongs
to $AA(\R \times \X)$ whenever it can be expressed as $F = G + \Phi,$
where $G\in AA(\R \times \X)$ and $\Phi \in PAP_0(\R \times \X)$. The
collection of such functions will be denoted by $PAA(\R \times \X)$.
\end{definition}

A substantial result is the next theorem, which is due to Xiao et
al. \cite{LL}.

\begin{theorem}\label{MN} \cite{LL} The space $PAA(\X)$ equipped with the sup
norm $\|\cdot\|_\infty$ is a Banach space.
\end{theorem}

\begin{theorem}\cite{LL}\label{CO} If $\Y$ is another Banach space, $f: \R \times \Y \mapsto \X$ belongs to $PAA(\R \times \X)$
and if $x \mapsto f(t,x)$ is uniformly continuous on each bounded
subset $K$ of $\Y$ uniformly in $t \in \R$, then the function defined
by $h(t) = f(t, \varphi(t))$ belongs to $PAA(\X)$ provided
$\varphi \in PAA(\Y)$.
\end{theorem}

 For
more on pseudo-almost automorphic functions and their generalizations, we refer
the reader to the recent book by Diagana \cite{TDbook}.

\section{Main Results}Fix $\alpha \in (0, 1)$.
To study the existence of pseudo-almost automorphic mild
solutions to Eq. (\ref{PR1}), we will need the following assumptions,

\begin{enumerate}
  \item [(H.1)] The linear operators $\{A(t)\}_{t \in \R}$ with constant common domains denoted $D$, satisfy the Acquistapace--Terreni conditions.
  
Let ${\U} = \{U(t,s): (t, s) \in \T\}$ denote the evolution family associated with the linear operators $A(t)$.

\item[(H.2)] The evolution family $U(t,s)$ is not only compact for $t > s$ but also is exponentially stable, i.e., there exists constants $N, \delta > 0$ such that $$\Big\|U(t,s)\Big\| \leq N e^{-\delta(t-s)}$$ for $t \geq s$. 

\item[(H.3)] The function $\R \times \R \mapsto \X_\alpha$, $(t,s) \mapsto U(t,s) \ph$, belongs to $bAA (\T, \X_\alpha)$ for all $\ph\in \X_\alpha$. 

\item[(H.4)] The linear operators $B(t) \in B(BC(\R, \X_\alpha), \X)$ for all $t \in \R$. Moreover, the following hold,
\begin{enumerate}
\item[a)] $\displaystyle C_0 := \sup_{t \in \R} \Big\|B(t)\Big\|_{B(BC(\R, \X_\alpha), \X)} \leq \frac{1}{2d(\alpha)},$ where $d(\alpha) := c(\alpha) (2\delta^{-1})^{1-\alpha} \Gamma(1-\alpha)$.
\item[b)] For all $\ph \in PAA(\X_\alpha)$, the function $t \mapsto B(t) \varphi$ belongs to $PAA(\X)$.
\end{enumerate}
\item[(H.5)] 
The function $f: \R \times \X_\alpha \mapsto \X$ is pseudo-almost automorphic in the first variable uniformly in the second one. For each bounded subset $K \subset \X_\alpha$, $f(\R, K)$ is bounded. Moreover, the function $u \mapsto f(t,u)$ is uniformly continuous on any bounded
subset $K$ of $\X_\alpha$ for each $t \in \R$. Finally, we suppose that there exists $L > 0$ such that
$$\sup_{t \in \R, \ \ \|\ph\|_\alpha \leq L} \Big\|f(t,\ph)\Big\| \leq \frac{L}{ 2d(\alpha)}.$$
\item[(H.6)] If $(u_n)_{n \in \N} \subset PAA(\X_\alpha)$ is uniformly bounded and uniformly convergent upon every compact subset of $\R$, then $f(\cdot, u_n(\cdot))$ is relatively compact in $BC(\R, \X)$.
\end{enumerate}

Consider the nonautonomous first-order differential equation,
\begin{eqnarray}\label{LT}
\frac{d\ph}{dt} = A(t) \ph + g(t), \ \ t \in \R,
\end{eqnarray}
where $g: \R \mapsto \X$ is a bounded continuous function.

\begin{definition}
Under assumption (H.1), a continuous function $\ph: \R \mapsto \X$ is said to be a mild
solution to Eq. (\ref{LT}) provided that
\begin{eqnarray}\label{M}
\ph(t)=U(t,s) \ph(s) + \int_{s}^{t}U(t,\tau) g(\tau)d\tau, \quad \forall (t,s) \in \T.
\end{eqnarray}
\end{definition}

\begin{lemma}\label{R}
Suppose {\rm (H.1)} holds and that ${\mathcal U} = \{U(t,s): (t,s) \in \T\}$ has exponential dichotomy with constants $N$ and $\delta$. If 
$g: \R \mapsto \X$ is a bounded continuous function, then $\ph$ given by
\begin{eqnarray}\label{VCF1} \ph(t) := \int_{-\infty}^\infty \Gamma (t,s) g(s) ds\end{eqnarray}
for all $t\in \R$, is the unique bounded mild solution to Eq. (\ref{LT}).
\end{lemma}

\begin{proof} The fact that $\ph$ given in Eq. (\ref{VCF1}) is a bounded mild solution to Eq. (\ref{LT}) is clear, see, e.g., \cite[Chap 4]{CH}. 
Now let $u, v$ be two bounded mild solutions to Eq. (\ref{LT}). Setting 
$w = u -v$, one can easily see that $w$ is bounded and that
$w(t) = U(t,s) w(s)$ for all $(t,s) \in \T$. 
Now using property (d) from exponential dichotomy (Definition \ref{expo}) it follows that $$P(t) w(t) = P(t) U(t,s) w(s) = U(t,s) P(s) w(s),$$and hence
\begin{eqnarray*}
\|P(t) w(t)\| & =&\|U(t,s) P(s) w(s)\| \\
&\leq& N e^{-\delta(t-s)}\|w(s)\|\\
&\leq& N e^{-\delta(t-s)}\|w\|_\infty, \ \ \forall (t,s) \in \T.\\
\end{eqnarray*} 

Now, given $t \in \R$ with $t \geq s$, if we let $s \to -\infty$, we then obtain that $P(t) w(t) = 0$, that is, $P(t) u(t) = P(t) v(t)$. Since $t$ is arbitrary it follows that $P(t) w(t) = 0$ for all $t \geq s$.

Similarly, from $w(t) = U(t,s) w(s)$ for all $t \geq s$ and  property (d) from exponential dichotomy (Definition \ref{expo}) it follows that $$Q(t) w(t) = Q(t) U(t,s) w(s) = U(t,s) Q(s) w(s),$$ and hence $U_Q(s,t) Q(t) w(t) = Q(s) w(s)$ for all $t \geq s$. Moreover,
\begin{eqnarray*}\|Q(s) w(s)\| &=& \|U_Q(s,t) Q(t) w(t)\| \\
&\leq& Ne^{-\delta(t-s)}\|w\|_\infty, \ \ \forall (t,s) \in \T.\\ 
\end{eqnarray*} 

Now, given $s \in \R$ with $t \geq s$, if we let $t \to \infty$, we then obtain that $Q(s) w(s) = 0$, that is, $Q(s) u(s) = Q(s) v(s)$. Since $s$ is arbitrary it follows that $Q(s) w(s) = 0$ for all $t \geq s$. The proof is complete.
\end{proof}

\begin{definition}
Under assumptions (H.1), (H.2), and (H.4) and if 
$f: \R \times \X_\alpha \mapsto \X$ is a bounded continuous function, then a continuous function $\ph: \R \mapsto \X_\alpha$ satisfying 
\begin{eqnarray}\label{VCF} \ph(t) = U(t,s) \ph(s) + \int_{s}^{t} U(t,s) \Big[B(s) \ph(s) + f(s, \ph(s))\Big] ds,\quad \forall (t,s) \in \T\end{eqnarray}
is called a mild solution to Eq. (\ref{PR1}).
\end{definition}

Under assumptions (H.1), (H.2), and (H.4) and if 
$f: \R \times \X_\alpha \mapsto \X$ is a bounded continuous function, it can be shown that the function $\ph: \R \mapsto \X_\alpha$ defined by 
\begin{eqnarray}\label{VCF} \ph(t) = \int_{-\infty}^{t} U(t,s) \Big[B(s) \ph(s) + f(s, \ph(s))\Big] ds\end{eqnarray}
for all $t\in \R$, is a mild solution to Eq. (\ref{PR1}).

Define the following integral operator,
$$(S \ph)(t) =\int_{-\infty}^{t}U(t,s) \Big[B(s) \ph(s) + f(s, \ph(s))\Big] ds.$$

We need the next lemma to establish the main results of the paper.

\begin{lemma}\label{B}Under assumptions {\rm (H.1)--(H.2)--(H.4)} and if 
$f: \R \times \X_\alpha \mapsto \X$ is a bounded continuous function, then the mapping $S: BC(\R,\X_\alpha) \mapsto BC(\R, \X_\alpha)$ is well-defined and continuous.

\end{lemma}

\begin{proof}
We first show that $S$ is well-defined and that $S (BC(\R, \X_\alpha)) \subset BC(\R, \X_\alpha)$. Indeed, letting $u \in BC(\R, \X_\alpha)$, $g(t) := f(t,u(t))$, and using Proposition \ref{pes},
we obtain
\begin{eqnarray*}
\big\|S u(t)\big\|_\alpha &\leq & \int_{-\infty }^t\big\|U(t,s) [B(s) u(t) + g(s)]
\big\|_\alpha ds\nonumber\\
& \leq & \int_{-\infty }^t
c(\alpha)e^{- \frac{\delta}{2}(t-s)}(t-s)^{-\alpha} 
\Big[\|B(s) u(s)\| + \|g(s)\|\Big]ds\nonumber\\
& \leq & \int_{-\infty }^t
c(\alpha)e^{- \frac{\delta}{2}(t-s)}(t-s)^{-\alpha} 
\Big[C_0\|u(s)\|_\alpha + \|g(s)\|\Big]ds\nonumber\\
& \leq & d(\alpha) \Big(C_0 \|u\|_{\alpha, \infty} + \|g\|_\infty\Big)\label{bound},
\end{eqnarray*} 
for all $t \in \R$, where $d(\alpha) =c(\alpha) (2\delta^{-1})^{1-\alpha} \Gamma(1-\alpha)$, and hence $Su: \R \mapsto \X_\alpha$ is bounded.

To complete the proof it remains to show that $S$ is continuous. For that, set$$F(s, u(s)) := B(s) u(s) + g(s) = B(s) u(s) + f(s, u(s)), \ \ \forall s \in \R.$$ 

Consider an arbitrary sequence of
functions
$u_n \in BC(\R, \X_\alpha)$ that converges uniformly to some $u \in BC(\R, \X_\alpha)$, that is,
$\big\|u_n -u\big\|_{\alpha, \infty} \to 0 \quad \mbox{as} \ \ n \to \infty.$

Now 
\begin{eqnarray*}
\big\|Su(t) - Su_n(t)\big\|_\alpha &=& \big\|\int_{-\infty}^t U(t,s) [F(s, u_n(s))-F(s, u(s))]\>ds\big\|_{\alpha}\\
&\leq& c(\alpha) \int_{-\infty}^t (t-s)^{-\alpha} e^{-\frac{\delta}{2}\>(t-s)}
\big\|F(s, u_n(s))-F(s, u(s))\big\|\>ds.\\
&\leq& c(\alpha) \int_{-\infty}^t (t-s)^{-\alpha} e^{-\frac{\delta}{2}\>(t-s)}
\big\|f(s, u_n(s))-f(s, u(s))\big\|\>ds\\
&+& c(\alpha) \int_{-\infty}^t (t-s)^{-\alpha} e^{-\frac{\delta}{2}\>(t-s)}
\big\|B(s) (u_n(s)- u(s))\big\|\>ds\\
&\leq& c(\alpha) \int_{-\infty}^t (t-s)^{-\alpha} e^{-\frac{\delta}{2}\>(t-s)}
\big\|f(s, u_n(s))-f(s, u(s))\big\|\>ds\\
&+& d(\alpha) C_0 \,\|u_n - u\|_{\alpha, \infty}.\\
\end{eqnarray*}

Using the continuity of the function $f: \R \times \X_\alpha \mapsto \X$ and the Lebesgue Dominated Convergence Theorem we conclude that
\begin{eqnarray*}
\big\|\int_{-\infty}^t U(t,s) P(s) [f(s, u_n(s))-f(s, u(s))]\>ds\big\| \to 0\>\>\>\mbox{as}\>\>\> n\to\infty.
\end{eqnarray*}
Therefore,
$\big\|S u_n - S u\big\|_{\alpha, \infty} \to 0$
as $n \to \infty$. The proof is complete.

\end{proof}

\begin{lemma}\label{POL} Under assumptions {\rm (H.1)---(H.5)}, then $S(PAA(\X_\alpha) \subset PAA(\X_\alpha)$.

\end{lemma}

\begin{proof} Let $u \in PAA(\X_\alpha)$ and define $h(s): = f(s, u(s)) + B(s) u(s)$ for all $s \in \R$. Using (H.5) and Theorem \ref{CO} it follows that the function $s \mapsto f(s, u(s))$ belongs to $PAA(\X)$.
Similarly, from (H.4), the function $s \mapsto B(s) u(s)$ belongs to $PAA(\X)$.
In view of the above, the function $s \mapsto h(s)$ belongs to $PAA(\X)$. Now write $h= h_1 + h_2 \in PAA (\X)$ where $h_1 \in AA (\X)$ and $h_2 \in PAP_0(\X)$ and set 
$$R h_j(t) :=  \int_{-\infty }^t
U(t,s) h_j(s)ds  \ \ \mbox{for all} \ \ t \in \R, \ \ j=1, 2.$$

Our first task consists of showing that $R \big(AA(\X)\big) \subset AA(\X_\alpha)$. Indeed,
using the fact that $h_1 \in AA (\X)$, for every
sequence of real numbers $(\tau'_n)_{n \in \N}$ there
   exist a subsequence $(\tau_n)_{n \in \N}$ and a function $f_1$ such that
      $$f_1 (t):=\lim_{n\to\infty} h_1 (t+\tau_n)$$
   is well defined for each $t\in\mathbb{R}$, and
      $$\lim_{n\to\infty} f_1 (t-\tau_n)= h_1 (t)$$
   for each $t\in \mathbb{R}$.

Now
\begin{eqnarray*}
(R h_1)(t + \tau_n) - (R f_1)(t) &=& \int_{-\infty}^{t+\tau_n} U(t
+\tau_n, s)
h_1(s)ds -  \int_{-\infty}^{t} U(t,s) f_1(s) ds\nonumber\\
&=& \int_{-\infty}^{t} U(t+\tau_n, s+\tau_n)
h_1(s+\tau_n)ds -\int_{-\infty}^{t} U(t,s)
f_1(s)ds.\nonumber\\
&=& \int_{-\infty}^{t} U(t+\tau_n, s+\tau_n)
\Big(h_1(s+\tau_n) - f_1(s)\Big)ds\nonumber\\
&+& \int_{-\infty}^{t} \Big(U(t+\tau_n, s+\tau_n) -  U(t,s)\Big) f_1(s) ds.\nonumber
\end{eqnarray*}

From Proposition \ref{pes} and the Lebesgue Dominated Convergence Theorem, it easily follows that
\begin{eqnarray*}
\Big\|\int_{-\infty}^{t} U(t+\tau_n, s+\tau_n) \Big(h_1(s+\tau_n) - f_1(s)\Big)ds \Big\|_\alpha &\leq& \int_{-\infty}^{t} \Big\|U(t+\tau_n, s+\tau_n) \Big(h_1(s+\tau_n) - f_1(s)\Big)\Big\|_\alpha ds \nonumber\\
&\leq&c(\alpha) \int_{-\infty}^{t} (t-s)^{-\alpha} e^{-\frac{\delta}{2}(t-s)}
\|h_1(s+\tau_n) - f_1(s)\| ds\\
&\to& 0 \ \ \mbox{as} \ \ n \to \infty.
\end{eqnarray*}

Similarly, from (H.3) and the Lebesgue Dominated Convergence Theorem, it follows 
\begin{eqnarray*}
\Big\|\int_{-\infty}^{t} ( U(t+\tau_n, s+\tau_n) - U(t,s)) f_1(s) ds\Big\|_\alpha &\leq& \int_{-\infty}^{t} \Big\|\Big(U(t+\tau_n, s+\tau_n) - U(t,s)\Big) f_1(s)\Big\|_\alpha ds\\
&\to& 0 \ \ \mbox{as} \ \ n \to \infty,
\end{eqnarray*}
and hence
$$(R f_1) (t) = \lim_{n \to \infty} (R h_1) (t+\tau_n)$$ for all $t \in
\R$.

Using similar arguments as above one obtains that 
$$(R h_1) (t) = \lim_{n \to \infty} (R f_1) (t-\tau_n)$$ for all $t \in
\R$, which yields, $t \mapsto (S h_1)(t)$ belongs to $AA(\X_\alpha)$.

The next step consists of showing that $R \big(PAP_0 (\X)\big) \subset PAP_0(\X_\alpha)$. Obviously, $R h_2 \in BC (\R, \X_\alpha)$ (see Lemma \ref{B}).
Using the fact that  $h_2 \in PAP_0(\X)$ and Proposition \ref{pes}  it can be easily shown that $(R h_2) \in PAP_0 (\X_\alpha)$. Indeed, for $r > 0$,
\begin{eqnarray} \displaystyle \frac{1}{2r} \int_{-r}^r \Big\|\int_{-\infty}^{t} U(t,s) h_2(s)ds\Big\|_\alpha  dt 
&\leq& \frac{c(\alpha)}{2r}\int_{-r}^{r}\int_0^{\infty}
e^{\frac{\delta}{2} s} s^{-\alpha} \Big\|h_2(t-s)\Big\|
dsdt\nonumber\\
&\leq& 
\int_0^{\infty} e^{\frac{\delta}{2} s} s^{-\alpha} \left(\frac{1}{2r}\int_{-r}^{r} \Big\|h_2(t-s)\Big\|
dt\right) ds.\nonumber
\end{eqnarray}

Using the fact that $PAP_0(\X)$ is translation-invariant it follows that
$$\displaystyle \lim_{r \to \infty} \frac{1}{2r}\int_{-r}^{r}
\Big\|h_2(t-s)\Big\| dt  = 0,$$ as $t \mapsto h_2(t-s) \in
PAP_0(\X)$ for every $s\in \R$. 

One completes the proof by
using the Lebesgue Dominated Convergence Theorem.
In summary, $(R h_2) \in  PAP_0 (\X_\alpha)$, which completes the proof.

\end{proof}

\begin{theorem}\label{AB}
Suppose assumptions {\rm (H.1)---(H.6)} hold, then Eq. (\ref{PR1}) has at least one pseudo-almost automorphic mild solution
\end{theorem}

\begin{proof}  Let $B_\alpha = \{u \in PAA(\X_\alpha): \|u\|_\alpha \leq L\}$. Using the proof of Lemma \ref{B} one can easily show that $B_\alpha$ is convex and closed. Moreover, from Lemma \ref{POL}, one can see that $S(B_\alpha) \subset PAA(\X_\alpha)$. 
Now for all $u \in B_\alpha$,
\begin{eqnarray*}
\big\|S u(t)\big\|_\alpha &\leq & \int_{-\infty }^t\big\|U(t,s) [B(s) u(t) + g(s)]
\big\|_\alpha ds\nonumber\\
& \leq & \int_{-\infty }^t
c(\alpha)e^{- \frac{\delta}{2}(t-s)}(t-s)^{-\alpha} 
\Big[\|B(s) u(s)\| + \|f(s, u(s))\|\Big]ds\nonumber\\
& \leq & \int_{-\infty }^t
c(\alpha)e^{- \frac{\delta}{2}(t-s)}(t-s)^{-\alpha} 
\Big[C_0\|u(s)\|_\alpha + \|f(s, u(s))\|\Big]ds\nonumber\\
& \leq & d(\alpha) \Big(C_0 L + \frac{L}{2d(\alpha)}\Big)\\
& \leq & d(\alpha) \Big(\frac{ L}{2d(\alpha)} + \frac{L}{2d(\alpha)}\Big)\\
&=& L
\end{eqnarray*} 
for all $t \in \R$ and hence $Su \in B_\alpha$. 

To complete the proof, we have to prove the following:
\begin{enumerate}
\item[a)] That $V = \{ Su(t): u \in B_\alpha\}$ is a relatively compact subset of $\X_\alpha$ for each $t \in \R$;
\item[b)] That $W = \{ Su: u \in B_\alpha\} \subset PAA(\X_\alpha)$ is equi-continuous.
\end{enumerate}

To show a), fix $t \in \R$ and consider an arbitrary $\varepsilon > 0$.

Now
\begin{eqnarray*}
(S_{\varepsilon} u)(t) &:=& \int_{-\infty}^{t-\varepsilon}U(t,s) F(s, u(s))ds, \ u \in B_\alpha\\
&=& U(t, t-\varepsilon) \int_{-\infty}^{t-\varepsilon} U(t-\varepsilon, s) F(s, u(s)) ds, \ u \in B_\alpha\\
&=& U(t, t-\varepsilon) (S u)(t-\varepsilon), \ u \in B_\alpha
\end{eqnarray*}
and hence $V_\varepsilon := \{ S_{\varepsilon} u(t): u \in B_\alpha\}$ is relatively compact in $\X_\alpha$ as the evolution family $U(t, t-\varepsilon)$ is compact by assumption.

Now

\begin{eqnarray*}
&&\big\|S u(t) - U(t, t-\varepsilon) \int_{-\infty}^{t-\varepsilon} U(t-\varepsilon, s) F(s, u(s)) ds\big\|_\alpha \\
&&\leq  \int_{t-\varepsilon}^t \|U(t,s) F(s, u(s)) \|_\alpha ds\\
&&\leq  c(\alpha)\int_{t-\varepsilon}^t e^{- \frac{\delta}{2}(t-s)}(t-s)^{-\alpha} \left\|F(s, u(s))\right\| ds\\
&&\leq c(\alpha) \int_{t-\varepsilon}^t (t-s)^{-\alpha} e^{-\frac{\delta}{2}\>(t-s)}
\big\|g(s)\big\|\>ds + c(\alpha) \int_{t-\varepsilon}^t (t-s)^{-\alpha} e^{-\frac{\delta}{2}\>(t-s)}
\big\|B(s) u(s)\big\|\>ds\\
&&\leq c(\alpha) \int_{t-\varepsilon}^t (t-s)^{-\alpha} e^{-\frac{\delta}{2}\>(t-s)}
\big\|f(s, u(s))\big\|\>ds
+ c(\alpha) C_0 \|u\|_{\alpha, \infty} \int_{t-\varepsilon}^t (t-s)^{-\alpha} e^{-\frac{\delta}{2}\>(t-s)}\>ds\\
&&\leq c(\alpha) L\Big(d^{-1}(\alpha) + C_0\Big) \int_{t-\varepsilon}^t (t-s)^{-\alpha} \>ds\\
&&= c(\alpha) L\Big(d^{-1}(\alpha) + C_0\Big)\varepsilon^{1-\alpha} (1-\alpha)^{-1},\\
\end{eqnarray*}
and hence the set $V := \{Su(t): u \in B_\alpha\} \subset \X_\alpha$ is relatively compact.

The proof for b) follows along the same lines as in Ding {\it et al.} \cite[Theorem 2.6]{ding} and hence is omitted. 

Now since $B_\alpha$ is a closed convex subset of $PAA(\X_\alpha)$ and that $S(B_\alpha) \subset B_\alpha$, it follows that $\overline{co}\,{S(B_\alpha)} \subset B_\alpha.$ Consequently, 
$S(\overline{co}\, S(B_\alpha)) \subset S(B_\alpha) \subset \overline{co}\ S(B_\alpha).$ Further, it is not hard to see that $\{u(t): u \in \overline{co}\,{S(B_\alpha)}\}$ is relatively compact in $\X_\alpha$
for each fixed $t \in \R$ and that functions in $\overline{co}\,{S(B_\alpha)}$ are equi-continuous on $\R$. 
Using Arzel\`a-Ascoli theorem, we deduce that the restriction of $\overline{co}\,{S(B_\alpha)}$ to any
compact subset $I$ of $\R$ is relatively compact in $C(I, \X_\alpha)$. In summary, $S: \overline{co}\,{S(B_\alpha)} \mapsto \overline{co}\,{S(B_\alpha)}$ is continuous and compact. 
Using the Schauder fixed point it follows that $S$ has a fixed-point, which obviously is a pseudo-almost automorphic
mild solution to Eq. (\ref{PR1}).

\end{proof}

In order to study Eq. (\ref{PR}), we need the following additional assumption:
\begin{enumerate}
\item[(H.7)] There exists a function $\rho \in L^1(\R, (0, \infty))$ with $\displaystyle \|\rho\|_{L^1(\R, (0, \infty))} \leq \frac{1}{2d(\alpha)}$ such that
$$\Big\|C(t) \ph\Big\| \leq \rho(t) \big\|\ph\big\|_{\alpha}$$ for all $\ph \in \X_\alpha$ and $t \in \R$.
\end{enumerate}

\begin{corollary}\label{ABB}
Suppose assumptions {\rm (H.1)--(H.2)--(H.3)--(H.5)-(H.6)-(H.7)} hold, then Eq. (\ref{PR}) has at least one pseudo-almost automorphic mild solution
\end{corollary}

\begin{proof} It suffices to check that (H.7) yields (H.4) in the case when the bounded linear operators $B(t)$ are defined by
$$B(t) \ph  := \int_{-\infty}^t C(t-s) \ph (s) ds$$for all $t \in \R$ and $\ph \in BC(\R,\X_\alpha)$ with $\X_\alpha = (\X, D)_{\alpha, \infty}$.

Indeed, since $\rho$ is integrable, it is clear that the operators $B(t)$ belong to $B(BC(\R, \X_{\alpha}), \X)$ for all $t \in \R$ with $\|B(t)\|_{B(BC(\R, \X_{\alpha}), \X)} \leq \|\rho\|_{L^1(\R, (0,\infty))}$. In fact, we can take $C_0 = \|\rho\|_{L^1(\R, (0,\infty))}$. To complete the proof, we have to show that the function $\R \mapsto \X$, $t \mapsto B(t) \ph$ is pseudo-almost automorphic for any $\ph \in PAA(\X_\alpha)$. For that, write $\ph = \ph_1 + \ph_2$, where $\ph_1 \in AA(\X_\alpha)$ and $\ph_2 \in PAP_0(\X_\alpha)$. 
Using the fact that the function $t \mapsto \ph_1(t)$ belongs to $AA (\X_\alpha)$, for every
sequence of real numbers $(\tau'_n)_{n \in \N}$ there
   exist a subsequence $(\tau_n)_{n \in \N}$ and a function $\psi_1$ such that
      $$\psi_1 (t):=\lim_{n\to\infty} \ph_1(t+\tau_n)$$
   is well defined for each $t\in\mathbb{R}$, and
      $$\lim_{n\to\infty} \psi_1 (t-\tau_n)= \ph_1 (t)$$
   for each $t\in \mathbb{R}$.

Now
\begin{eqnarray*}
B (t + \tau_n)\ph_1 - B (t) \psi_1 &=& \int_{-\infty}^{t+\tau_n} C(t
+\tau_n - s)
\ph_1 (s)ds -  \int_{-\infty}^{t} C(t-s) \psi_1(s) ds\nonumber\\
&=& \int_{-\infty}^{t} C(t -s)
\ph_1(s+\tau_n)ds -\int_{-\infty}^{t} C(t-s)
\psi_1(s)ds\nonumber\\
&=& \int_{-\infty}^{t} C(t-s)
\Big(\ph_1(s+\tau_n) - \psi_1(s)\Big)ds\nonumber\\
\end{eqnarray*}
and hence
\begin{eqnarray*}\Big\|B(t + \tau_n)\ph_1 - B(t)\psi_1\Big\| &\leq& \Big\|\int_{-\infty}^{t} C(t-s)
\Big(\ph_1(s+\tau_n) - \psi_1(s)\Big)ds\Big\|\\
&\leq& \int_{-\infty}^{t} \Big\|C(t-s)
\Big(\ph_1(s+\tau_n) - \psi_1(s)\Big)\Big\| ds\\
&\leq& \int_{-\infty}^{t} \rho(t-s)
\Big\|\ph_1(s+\tau_n) - \psi_1(s)\Big\|_\alpha ds\\
\end{eqnarray*}
which by Lebesgue Dominated Convergence Theorem yields
$$\lim_{n \to \infty} \Big\|B(t + \tau_n)\ph_1 - B(t) \psi_1\Big\| = 0.$$

Using similar arguments, we obtain
$$\lim_{n \to \infty} \Big\|B(t - \tau_n)\psi_1 - B(t) \ph_1\Big\| = 0.$$

For $r > 0$,
\begin{eqnarray} \displaystyle \frac{1}{2r} \int_{-r}^r \Big\|\int_{-\infty}^{t} C(t-s) \ph_2(s)ds\Big\|  dt 
&\leq& \frac{1}{2r}\int_{-r}^{r}\int_0^{\infty}
\rho(s) \Big\|\ph_2(t-s)\Big\|_\alpha
dsdt\nonumber\\
&\leq& 
\int_0^{\infty} \rho(s) \left(\frac{1}{2r}\int_{-r}^{r} \Big\|\ph_2(t-s)\Big\|_\alpha
dt\right) ds.\nonumber
\end{eqnarray}

Now using the translation invariance of the space $PAP_0(\X_\alpha)$, it follows that,
$$\displaystyle \lim_{r \to \infty} \frac{1}{2r}\int_{-r}^{r}
\Big\|\ph_2(t-s)\Big\|_\alpha dt  = 0,$$ as $t \mapsto \ph_2(t-s) \in
PAP_0(\X_\alpha)$ for every $s\in \R$. 

Therefore,
$$\lim_{r \to \infty} \frac{1}{2r} \int_{-r}^r \Big\|\int_{-\infty}^{t} C(t-s) \ph_2(s)ds\Big\|  dt = 0$$
by
using the Lebesgue Dominated Convergence Theorem.

\end{proof}

\section{Example}
Fix $\alpha \in (0, 1)$. Let $\Omega \subset \R^N$ be an open bounded set with smooth boundary $\partial \Omega$ and
let $\X = L^2(\Omega)$ be equipped with its natural norm
$\|\cdot\|_{L^2(\Omega)}$ defined for all $\varphi \in L^2(\Omega)$ by
$$\|\varphi\|_{L^2(\Omega)} = \Bigg(\int_{\Omega} \|\varphi(s)\|^2 ds\Bigg)^{1/2}.$$

Motivated by natural phenomena such as population genetics \cite{V2, V3} or nerve pulse propagation \cite{V1}, in this section we study the existence of pseudo-almost automorphic solutions to the parabolic Neumann boundary value problem given by,
\begin{eqnarray}\label{D1000} \hspace{.6cm} \frac{\partial\ph}{\partial t}(t,x) &=& a(t) \Delta\ph(t,x) + \int_{-\infty}^t b(t-s) \ph (s,x) ds + f(t, \ph(t,x))\\&+& \eta a(t) \ph(t,x), \ \  (t,x)\in \R \times \Omega,\nonumber
\end{eqnarray}
\begin{eqnarray}\label{D2000}
\frac{\partial \ph}{\partial n}(t,x) &=& 0, \ \ (t,x)\in \R \times \partial \Omega,\end{eqnarray}
where $\eta > 0$ is a constant, $a: \R \mapsto \R$ and $b: \R \mapsto (0, \infty)$ are functions, the function $f: \R \times L^2(\Omega) \mapsto L^2(\Omega)$ is pseudo-almost automorphic in $t \in \R$ uniformly in $\varphi \in L^2(\Omega)$, and $\Delta$ stands for the usual Laplace operator in the space variable $x$.

Letting
$$A_\eta(t) \ph = a(t) (\Delta +\eta) \ph \ \ \mbox{for all} \ \ \ph \in D(A_\eta(t)) = D(\Delta) = \Big\{\ph \in H^2 (\Omega): \frac{\partial \ph}{\partial n} = 0 \ \ \mbox{on} \ \ \partial \Omega\Big\},$$
$$C(t)\ph = b(t) \ph \ \ \mbox{for all} \ \ \ph \in D(C(t)) = D(\Delta),$$ and $f(t,\ph) = h(t,\ph)$, one can easily see that Eq. (\ref{PR}) is exactly
the nonautonomous parabolic Neumann boundary value problem
formulated in 
Eqs. (\ref{D1000})-(\ref{D2000}).

This setting requires the following assumptions,
\begin{enumerate}
\item[(H.8)] The function $a: \R \mapsto (0, \infty)$ is almost automorphic with $$\displaystyle \inf_{t \in \R} a(t) = a_0 > 0.$$
\item[(H.9)] The function $b: \R \mapsto (0, \infty)$ belongs to $L^1(\R, (0, \infty))$ with $$\displaystyle \int_{-\infty}^\infty b(s) ds \leq \frac{1}{2\widetilde{C} d(\alpha)}$$where
$\widetilde{C}$ is the bound of the continuous injection $L_\alpha^2(\Omega):= (L^2(\Omega), D(\Delta))_{\alpha, \infty} \hookrightarrow L^2(\Omega)$.
\end{enumerate}

Setting
$A_\eta \ph = - (\Delta +\eta) \ph$ for all $\ph \in D(A_\eta)=D(\Delta),$ one can easily see that $A_\eta (t) = - a(t) A_\eta$. Of course, $-A_\eta$ is a sectorial operator on $L^2(\Omega)$. Let $(T(t))_{t \geq 0}$ be the analytic semigroup generated by the operator $-A_\eta$. It is well-known that the semigroup $T(t)$ is not only compact for $t > 0$ but also is exponentially stable as
$$\|T(t)\|  \leq  e^{-\eta t}$$ for all $t \geq 0$.

In view of the above,
it is now clear that the evolution family $U(t,s)$ associated with $A_\eta(t)$ is given by 
$$U(t,s) = T\Bigg(\int_{s}^{t} a(r)dr\Bigg)$$
 for all $t \geq s, \ t, s \in \R$
as $$\displaystyle U(t,s)-U(\tau,s)= \Big[ T\Big(\int_{\tau}^{t} a(r)dr\Big)-I\Big] T\Big(\int_{s}^{\tau} a(r)dr\Big)$$for all $t>\tau$ and $t, \tau \in \R$.

 Further, under assumptions (H.8), the compactness of the semigroup and the exponential stability of $T(t)$ it follows that the evolution family $U(t,s)$ is not only compact for $t > s$ but also is exponentially stable as
$$\|U(t,s)\| \leq e^{-\eta a_0 (t-s)}$$
for all $t,s \in \R$ with $t \geq s$.

Additionally, the functions $(t,s) \mapsto U(t,s) \ph$, $\R \times \R \mapsto L_\alpha^2(\Omega)$ belongs to $bAA (\T, L_\alpha^2(\Omega))$ for all $\ph\in L_\alpha^2(\Omega)$.

Now

\begin{eqnarray*}
\|C(t) \varphi\|_{L^2(\Omega)} &=& b(t) \|\varphi\|_{L^2(\Omega)}\\
&\leq&  \widetilde{C}b(t) \|\ph \|_{\alpha}\\
\end{eqnarray*}
for all $\displaystyle \ph \in L_{\alpha}^2(\Omega)$ and $t \in \R$.

In view of the above, it is clear that assumptions (H.1)-(H.2)-(H.3)-(H.4)-(H.7) are fulfilled. Therefore, using Corollary \ref{ABB}, we obtain the following theorem.

\begin{theorem}\label{ABC}
Under assumptions {\rm (H.5)--(H.6)--(H.8)--(H.9)}, then the system  Eqs. (\ref{D1000})-(\ref{D2000}). has at least one pseudo-almost automorphic mild solution
\end{theorem}

\bibliographystyle{amsplain}

\begin{thebibliography}{10}
\bibitem{aq}
P. Acquistapace, Evolution operators and strong solutions of abstract linear parabolic equations. {\it Differential Integral
Equations} {\bf 1} (1988), no. 4, pp. 433--457.


\bibitem{at}
P. Acquistapace and B. Terreni, A unified approach to abstract linear nonautonomous parabolic equations. {\it Rend. Sem.
Mat. Univ. Padova} {\bf 78} (1987), pp. 47--107.


\bibitem{adez}
M. Adimy and K. Ezzinbi, A class of linear partial neutral functional-differential equations with nondense domain. {\it J. Diff. Eqns.} {\bf 147} (1998), pp. 285--332.

\bibitem{AG}
R.P. Agarwal, T. Diagana, and E. Hern\`andez,
Weighted pseudo almost periodic solutions to some partial neutral functional differential equations. {J. Nonlinear Convex Anal.} {\bf 8} (2007), no. 3, pp. 397--415.


\bibitem{Am}
H. Amann, \textit{Linear and quasilinear parabolic problems}.
Birkh\"{a}user, Berlin 1995.


\bibitem{AA}
M. Anguiano, T. Caraballo, J. Real, and J. Valero, Pullback attractors for a nonautonomous integro-differential equation with memory in some unbounded domains. {\it Internat. J. Bifur. Chaos Appl. Sci. Engrg}. {\it 23} (2013), no. 3, 1350042, 24 pp.


\bibitem{W}
M. Baroun, S. Boulite, T. Diagana, and L. Maniar,  Almost periodic
solutions to some semilinear non-autonomous thermoelastic plate
equations. {\it J. Math. Anal. Appl.} {\bf  349}(2009), no. 1,
pp. 74--84.

\bibitem{CC}
J. C. Chang, Solutions to non-autonomous integrodifferential equations with infinite delay. {\it J. Math. Anal. Appl.} {\bf 331} (2007), no. 1, pp. 137--151.

\bibitem{C}
G. Chen and R. Grimmer, Ronald, Integral equations as evolution equations. {\it J. Differential Equations} {\bf 45} (1982), no. 1, pp. 53--74.

\bibitem{CH}
C. Chicone and Y. Latushkin, {\it Evolution semigroups in dynamical systems and differential equations}.
Mathematical Survey and Monographs, Vol. {\bf 70}, Amer. Math. Soc., 1999.

\bibitem{cor}
C. Corduneanu, {\it Almost periodic functions}.
AMS Chelsea Publishing, 1989.

\bibitem{da1}
G. Da Prato and M. Iannelli, Existence and regularity for a class of integro-differential equations of parabolic type, {\it J. Math. Anal. Appl.} {\bf 112} (1985), no. 1, pp. 36--55. 

\bibitem{da2}
G. Da Prato and A. Lunardi, Solvability on the real line of a class of linear Volterra integro-differential equations of parabolic type, {\it Ann. Mat. Pura Appl.} (4) {\bf 150} (1988), pp. 67--117. 


\bibitem{da3}
G. Da Prato and A. Lunardi, Hopf bifurcation for nonlinear integro-differential equations in Banach spaces with infinite delay. {\it Indiana Univ. Math. J.} {\bf 36} (1987), no. 2, pp. 241--255.

\bibitem{TDbook}
  T. Diagana, \emph{Almost automorphic type and almost periodic type functions in abstract spaces.} Springer, 2013, New York, XIV, 303 p.


\bibitem{MCM}
T. Diagana, Almost periodic solutions for some higher-order nonautonomous differential equations with operator coefficients. {\it Math. Comput. Modelling} {\bf 54} (2011), no. 11-12, pp. 2672--2685.


\bibitem{TT}
T. Diagana, {\it Pseudo almost periodic functions in Banach
spaces}. Nova Science Publishers, Inc., New York, 2007.

\bibitem{d6}
T. Diagana, Existence of pseudo almost periodic solutions to some
classes of partial hyperbolic evolution equations. {\it Electron.
J. Qual. Theory Differ. Equ.} 2007, No. 3, 12 pp.

\bibitem{PAMS}
T. Diagana, Almost periodic solutions to some second-order nonautonomous differential equations. {\it Proc. Amer. Math. Soc.} {\bf 140} (2012), pp. 279--289.

\bibitem{DE}
T. Diagana, H. R. Henriquez, and E. M. Hern\`{a}ndez, Almost automorphic mild solutions to some partial neutral functional-differential equations and applications. {\it Nonlinear Anal.} {\bf 69} (2008), no. 5-6, pp. 1485--1493.


\bibitem{DE1}
T. Diagana, E. Hern\`{a}ndez, J. P. C. dos Santos, Existence of asymptotically almost automorphic
solutions to some abstract partial neutral integro-differential equations. {\it Nonlinear Anal.} {\bf 71} (2009), no. 1, pp. 248--257.

\bibitem{DE2}
T. Diagana, E. Hern\`{a}ndez, Existence and uniqueness of pseudo almost periodic solutions to some abstract partial neutral functional-differential equations and applications. {\it J. Math. Anal. Appl.} {\bf 327} (2007), no. 2, pp. 776--791.

\bibitem{ding}
H. S. Ding, J. Liang, G. M. N'Gu\'er\'ekata, T. J. Xiao, Mild pseudo-almost periodic solutions of nonautonomous semilinear evolution equations. {\it Math. Comput. Modelling} {\bf 45} (2007), no. 5-6, pp. 579--584.

\bibitem{F1}
A. Friedman, Monotonicity solution of Volterra integral equations in Banach space,
{\it Trans. Amer. Math. Soc.} {\bf 138} (1969), pp. 129--148.

\bibitem{F2}
A. Friedman and M. Shinbrot, Volterra integral equations in Banach space. {\it Trans.
Amer. Math. Soc.} {\bf 126} (1967), pp. 131--179.



\bibitem{HR}
M. L. Heard and S. M. Rankin III, A semilinear parabolic Volterra integro-differential equation. {\it J. Differential Equations} {\bf 71} (1988), no. 2, pp. 201--233.

\bibitem{G}
M.E. Gurtin and A.C. Pipkin, A general theory of heat conduction with infinite wave speed. {\it Arch. Rat. Mech. Anal}. {\bf 31} (1968), pp. 113--126.


\bibitem{H0}
E. Hern\`andez M.,
$C^\alpha$-classical solutions for abstract non-autonomous integro-differential equations. {\it Proc. Amer. Math. Soc.} {\bf 139} (2011), pp. 
4307--4318.

\bibitem{H1}
E. Hern\`andez M., M. L. Pelicer, and J. P. C. dos Santos,
Asymptotically almost periodic and almost periodic solutions for a
class of evolution equations, {\it Electron. J. Differential Equations} {\bf
2004}(2004), no. 61, pp. 1--15.


\bibitem{H2}
E. Hern\`andez M. and J. P. C. dos Santos, Asymptotically almost periodic and almost periodic solutions for a class of partial integrodifferential equations. {\it Electron. J. Differential Equations} (2006), No. 38, 8p. 

\bibitem{li}
H. X. Li, F. L. Huang, and J. Y. Li, Composition of pseudo almost-periodic functions and semilinear differential equations. {\it J. Math. Anal. Appl.} {255} (2001), no. 2, pp. 436--446.

\bibitem{L}
J. Liang, J. Zhang, and T-J. Xiao, Composition of pseudo almost
automorphic and asymptotically almost automorphic functions. {\it
J. Math. Anal. Appl.} {\bf 340} (2008), pp. 1493--1499.


\bibitem{LLL}
J. Liang, G. M. N'Gu\'er\'ekata, T-J. Xiao, and J. Zhang, Some
properties of pseudo almost automorphic functions and applications
to abstract differential equations. {\it Nonlinear Anal.} {\bf 70}
(2009), no. 7, pp. 2731--2735.

\bibitem{J}
J. H. Liu, Integrodifferential equations with non-autonomous operators. {\it Dynam. Systems Appl.} {\bf 7} (1998), no. 3, pp. 427--439.


\bibitem{li1}
C. Lizama and R. Ponce, Almost automorphic solutions to abstract Volterra equations on the line. {\it Nonlinear Anal.} {\bf 74} (2011), no. 11, pp. 3805--3814.

\bibitem{li2}
C. Lizama and G. M. N'Gu\'er\'ekata, Bounded mild solutions for semilinear
integro differential equations in Banach spaces. {\it Integral Equations Operator Theory} {\bf 68} (2010), no. 2, pp. 207--227.

\bibitem{Lun}
 A. Lunardi, {\it Analytic semigroups and optimal regularity in parabolic problems}. Progress in Nonlinear Differential Equations and their Applications, 16. Birkh\"{a}user Verlag, Basel, 1995. 

\bibitem{Lun2}
A. Lunardi, Regular solutions for time dependent abstract integro-differential equations with singular kernel. {\it J. Math. Anal. Appl.} {\bf 130} (1988), no. 1, pp. 1--21

\bibitem{Lun4}
A. Lunardi and E. Sinestrari, $C\sp \alpha$-regularity for nonautonomous linear integro-differential equations of parabolic type. {\it J. Differential Equations} {\bf 63} (1986), no. 1, pp. 88--116.

\bibitem{M}
L. Maniar and A. Rhandi, Nonautonomous retarded wave equations. {\it J. Math. Anal. Appl.} {\bf 263} (2001), no. 1, pp. 14--32.


\bibitem{NO}
J. A. Nohel, {\it Nonlinear Volterra equations for heat flow in materials with memory}.
MRC Rech. Summary Report $\#$2081, Madison, WI.

\bibitem{N}
J. W. Nunziato, On heat conduction in materials with memory. {\it Quart. Appl. Math.} {\bf 29} (1971), pp. 187--204.

\bibitem{pruss}
J. Pr\"{u}ss, Evolutionary integral equations and applications, Monographs in Mathematics, vol. 87, Birkh\"{a}user Verlag, Basel, 1993. 

\bibitem{V1}
P. A. Vuillermot, Global exponential attractors for a class of almost-periodic parabolic equations in ${\bf R}\sp N$. {\it Proc. Amer. Math. Soc.} {\bf 116} (1992), no. 3, pp. 775--782.

\bibitem{V2}
P. A. Vuillermot, Almost-periodic attractors for a class of nonautonomous reaction-diffusion equations on ${\bf R}\sp N$. II. Codimension-one stable manifolds. {\it Differential Integral Equations} {\bf 5} (1992), no. 3, pp. 693--720.

\bibitem{V3}
P. A. Vuillermot, Almost periodic attractors for a class of nonautonomous reaction-diffusion equations on ${\bf R}\sp N$. I. Global stabilization processes. {\it J. Differential Equations} {\bf 94} (1991), no. 2, pp. 228--253. 

\bibitem{LL}
T. J. Xiao, J. Liang, J. Zhang, Pseudo almost automorphic solutions
to semilinear differential equations in Banach spaces. {\it
Semigroup Forum} {\bf 76} (2008), no. 3, pp. 518--524.

\bibitem{XJ}
T. J. Xiao, X-X. Zhu, J. Liang, Pseudo-almost automorphic mild
solutions to nonautonomous differential equations and
applications. {\it Nonlinear Anal.} {\bf 70} (2009), no. 11,
pp. 4079--4085.

\bibitem{webb}
G. F. Webb, An abstract semilinear Volterra integrodifferential equation, {\it Proc. Amer. Math. Soc.} {\bf 69} (1978), no. 2, pp. 255--260






\end{thebibliography}

\end{document}